\documentclass[10pt,a4paper]{amsart}

\usepackage[latin2]{inputenc}
\usepackage[T1]{fontenc}

\usepackage{euscript}
\usepackage{eufrak}
\usepackage{amsmath}
\usepackage{amsthm}
\usepackage{amssymb}
\usepackage[matrix,arrow]{xy}

\usepackage{enumerate}

\IfFileExists{p2e-dvips.def}{
\usepackage{pict2e}
}{}

\usepackage[bookmarksnumbered,bookmarksopen,unicode]{hyperref}
\hypersetup{
  pdftitle={},
  pdfauthor={G\'abor Braun and Andr\'as N\'emethi},
  pdfkeywords={},
  pdfsubject={MSC2000 Primary 32S05, 32S25, 57M27; Secondary
    32S45, 32S50, 32C35, 57R57},
}

\numberwithin{equation}{section}
\numberwithin{equation}{subsection}

\theoremstyle{plain}
\newtheorem{theorem}[equation]{Theorem}
\newtheorem{lemma}[equation]{Lemma}
\newtheorem{proposition}[equation]{Proposition}
\newtheorem{corollary}[equation]{Corollary}
\newtheorem{conjecture}[equation]{Conjecture}

\theoremstyle{definition}
\newtheorem{definition}[equation]{Definition}
\newtheorem{example}[equation]{Example}
\newtheorem{remark}[equation]{Remark}

\newtheorem{bekezd}[equation]{}

\newcommand*{\ssw}[1]{\operatorname{sw}\sb{#1}}
\newcommand*{\Pic}[1]{\operatorname{Pic} ( #1 )}

\newcommand{\Spinc}{\texorpdfstring{%
    \ifmmode Spin\textsuperscript{c}\fi\else \operatorname{Spin\sp{c}}\fi}%
  {Spin-c}}

\newcommand{\cO}{\mathcal{O}}
\newcommand{\cH}{\mathcal{H}}
\newcommand*{\linebundle}{\mathcal{L}}
\newcommand{\cV}{\mathcal{V}}
\newcommand{\cW}{\mathcal{W}}
\newcommand{\cF}{\mathcal{F}}
\newcommand{\cP}{\mathcal{P}}
\newcommand{\cPr}{\mathcal Pr}
\newcommand{\cS}{\mathcal{S}}
\newcommand{\cE}{\mathcal{E}}
\newcommand{\cG}{\mathcal{G}}

\newcommand{\cB}{\mathcal{B}}
\newcommand{\cI}{\mathcal{I}}
\newcommand{\cJ}{\mathcal{J}}

\newcommand{\setC}{\mathbb{C}}
\newcommand{\setQ}{\mathbb{Q}}
\newcommand{\setR}{\mathbb{R}}
\newcommand{\setZ}{\mathbb{Z}}

\newcommand{\setP}{\mathbb{P}}

\newcommand{\gs}{\mathfrak{s}}
\newcommand{\hh}{\mathfrak{h}}
\newcommand{\bt}{{\bf t}}

\newcommand{\wxb}{\widetilde{X}(\bar{E})}

\providecommand{\coloneqq}{\mathrel{:=}}

\newcommand*{\twistedbundle}[3][]{\mathinner{{#2}\sb{#1}}%
    {\left(#3\right)}}

\newcommand*{\bundlefromdivisor}[2][]{\twistedbundle[{#1}]{%
    \mathcal{O}}{#2}}

\newcommand*{\smartfrac}[2]{\mathchoice
         \frac{#1}{#2}   
  {\mathinner{{#1}/{#2}}}
  {\mathinner{{#1}/{#2}}}
  {\mathinner{{#1}/{#2}}}
}

\begin{document}

\title{The cohomology of line bundles of splice-quotient singularities}

\date{}

\author{Andr\'as N\'emethi}
\thanks{}
\address{R\'enyi Institute of Mathematics\\
  Budapest\\
  Re\'altanoda u. 13\textendash15\\
  1053\\
 Hungary}
\email{nemethi@renyi.hu}



\begin{abstract}
We provide several results on splice-quotient singularities: (1) a
combinatorial expression of the dimension of the first cohomology
of all `natural' line bundles (for definitions see (\ref{ss:11}));
(2) an equivariant Campillo\textendash Delgado \textendash
Gusein-Zade type formula about the dimension of relative sections
of line bundles,  extending former results about rational and
minimally elliptic singularities;  (3)  in particular, we prove
that the equivariant, divisorial multi-variable Hilbert series is
topological;  (4)  a combinatorial description of  divisors of
analytic function-germs; (5) and  an expression for  the
multiplicity of the singularity from its resolution graph (in
particular solving Zariski's Multiplicity Conjecture for
splice-quotient hypersurfaces).

Additional, in (\ref{re:SW}) we establish a new formula for the Seiberg-Witten
invariants of {\em any} rational homology sphere singularity link,
whose validity for links of splice-quotients is the consequence of the 
above results (and the proof of the
general case will appear in \cite{NSW}).   
\end{abstract}

\maketitle

\section{Introduction}\label{sec:introduction}

Splice-quotient singularities were introduced by Neumann and Wahl
 and became the subject of an intense mathematical
activity
\cite{MR1900786,MR1981612,NWuj2,NWuj,Graph,Wahl2006,OkumaRat,Ouac-c,Opg,NO1,NO2,BN,Stev}.
In their definition, one starts with a resolution graph $\Gamma$
which is a tree, and the genera of all the vertices are zero, and
it also satisfies an additional combinatorial restriction. 
From $\Gamma$  one constructs an equisingular
family of singularities, as analytic realizations supported by the
topological type fixed by $\Gamma$. Rational and minimally
elliptic singularities  \cite{Ouac-c}  and weighted-homogeneous
singularities \cite{Neu} are examples of splice-quotients.

In this article we settle some basic results regarding these
singularities. The fact that they are `canonically' constructed from a
combinatorial object, gives hope that their basic discrete analytic
invariants should be representable from the topology. The present
article unifies apparently different research directions working on
these type of questions for different invariants.

First, we show that the equivariant divisorial multi-variable
Hilbert series \cite{CDGEq,CDGb} (see also \cite{CHR,CDG} for the
non-equivariant version) is topological. In order to do this, we
prove the equivariant Campillo--Delgado--Gusein-Zade formula for
splice-quotients, cf.  (\ref{th:21}). This
identifies a series $\cP$, which codifies the same information as
the Hilbert series itself, with a combinatorially defined  `zeta
function'. (Here, for the equivariant divisorial Hilbert series
and for the Campillo--Delgado--Gusein-Zade formula, we adopt the
formalism established in \cite{CDGb}.)

Such a formula was proved  by Campillo, Delgado and Gusein-Zade
for rational singularities \cite{CDG,CDGEq}, by a different method
it was reproved for rational and extended for minimal elliptic
singularities  in \cite{CDGb}, but it was unknown (at least by the
author) even for weighted-homogeneous singularities. For such a
singularity with a $\setC^*$-action, its reduction to a
one-variable identity (corresponding to the valuation of the
`central' divisor) is the celebrated  Pinkham-Dolgachev-Neumann
formula, which identifies the ($\setC^*$-equivariant) Poincar\'e
series with a topologically defined zeta-function
\cite[(5.7)]{Pi1}, \cite{Neu}.


The Campillo--Delgado--Gusein-Zade formula, surprisingly enough,
can be connected with the (equivariant) Seiberg-Witten Invariant
Conjecture for `natural' line bundles $\linebundle$ (see
\cite{nemethi02:_seiber_witten,Line} for the conjecture and
\cite{INV} for related results). This provides a topological
description of the dimensions $h^1(\linebundle)$ of the cohomology
groups (including the geometric genus corresponding to the trivial
line bundle) involving the Seiberg-Witten invariant of the link,
cf. (\ref{cor:QS}). Although  \cite{BN} provides a proof 
of  the conjecture for splice-quotients, that version,
as the original conjecture itself, contains a restriction
regarding the Chern class of the involved line bundles. Without
this restriction, the topological candidate for
$h^1(\linebundle)$, even at conjectural level, was not known. Here
we treat the complete general case without any restriction on the
Chern class. The `correction term' needed for arbitrary Chern
class is provided by a `truncated part' of the expression
appearing in Campillo--Delgado--Gusein-Zade formula.

In the last section we characterize from the graph the local
divisors of analytic functions for splice-quotient singularities
(a very difficult open task, in general), and we determine topologically
the multiplicity of the singularity too. In particular, we answer
positively Zariski's Multiplicity Conjecture for all
splice-quotients.

\section{Notations and motivations}
\label{sec:main-results}

\subsection{General surface singularities}\label{ss:11}

Let \((X,o)\) be a complex normal surface singularity whose link
is a rational homology sphere.  Let \(\pi:\widetilde{X}\to X\) be
a good resolution with dual graph  \(\Gamma\) whose vertices are
denoted by $\cV$. Recall that the link is a rational homology
sphere if and only if \(\Gamma\) is a tree and all the irreducible
exceptional divisors have genus \(0\).

Set \(L \coloneqq H_2 ( \widetilde{X},\setZ )\). It is freely
generated by the classes of the irreducible exceptional curves
\(\{E_v\}_{v\in\cV}\). They will also be identified with the
integral cycles supported on $E=\pi^{-1}(o)$. If  \(L'\) denotes
\(H^2( \widetilde{X}, \setZ )\), then the intersection form
\((\,,\,)\) on \(L\) provides an embedding \(L \hookrightarrow
L'\) with factor the first homology group \(H\) of \(\partial
\widetilde{X}\) (or, of the link). Moreover,  $(\,,\,)$ extends to
$L'$. $L'$ is freely generated by the duals \(E_v^*\), where we
prefer the convention $ ( E_v^*, E_w) =  -1 $ for $v = w$, and
$=0$ otherwise.

Effective classes \(l=\sum r_vE_v\in L'\) with all
\(r_v\in\setQ_{\geq 0}\) are denoted by \(L'_{\geq 0}\), and
$L_{\geq 0}:=L'_{\geq 0}\cap L$. Moreover, denote by $\cS'$ the
anti-nef cone $\{l'\in L'\,:\, (l',E_v)\leq 0 \ \mbox{for all
$v$}\}$.  It is generated over $\setZ_{\geq 0}$ by the
base-elements $E_v^*$. Since  {\em all the entries } of $E_v^*$
are {\em strict} positive, $\cS'$ is a sub-cone of $L_{\geq 0}'$,
and for any fixed $a\in L'$ the set
\begin{equation}\label{eq:finite}
\{l'\in \cS'\,:\, l'\ngeq a\} \ \ \mbox{is finite}.
\end{equation}
Set \(Q:=\{\sum l'_vE_v\in L', \ 0\leq l_v' < 1\}\).
For any $l'\in L'$ write its class in $H$ by $[l']$, and let
$r_{[l']}\in L'$  be its unique representative  in $Q$
with $l'-r_{[l']}\in L$. Finally,
denote by $\theta:H\to \widehat{H}$ the isomorphism $[l']\mapsto
e^{2\pi i(l',\cdot)}$ of $H$ with its Pontrjagin dual
$\widehat{H}$.


\vspace{2mm}

Next we list three major motivations of our investigation with some
interlocking  connecting them.

\subsection{First motivation: cohomology of line
bundles.}\label{FM}
 Most of the analytic geometry of \(\widetilde{X}\) (hence of $(X,o)$ too)
is described by its line bundles and their  cohomology groups. One
has the following basic goal (below the cohomology groups are
considered on $\widetilde{X}$):
%
 {\em For any \(\linebundle\in\Pic{\widetilde{X}}\)
 and effective cycle $l\in L_{\geq 0}$ recover the dimensions
\begin{equation}\label{eq:00}
 (a)  \ \dim \smartfrac{H^0 \left(\linebundle  \right)}{H^0
\left( \linebundle \left( - l \right)
 \right)}\ \ \ \mbox{and} \ \    (b) \ \dim H^1(\linebundle)
 \end{equation}
 from the combinatorics of \,$\Gamma$} (at least for some families of
 singularities).

Thanks to a recent intense activity, there is an increasing optimism to
understand this problem for  `special' line  bundles.
They  are provided by the splitting of the
cohomological exponential exact sequence
summarized below, cf.  \cite[\S 3]{Line}, \cite[(2.2)]{OkumaRat}:
\begin{center}
\ \xymatrix{
  &&& L \ar[d] \ar@{-->}[dl]\\
  0 \ar[r]& {H^1 ( \widetilde{X}, {\mathcal{O}}_{\widetilde{X}} )} \ar[r]&
  \Pic{\widetilde{X}} \ar@<.5ex>[r]^-{c_1 }& L' \ar[r]
  \ar@<.5ex>@{-->}[l]^-{\mathcal{O}}& 0.
}
\end{center}
The first Chern class \(c_1\) is a surjective and it has an
obvious section on the subgroup \(L\): it maps every element to its associated
line bundle. This section has a unique extension \(\mathcal{O}\) to
\(L'\). We call a line bundle \emph{natural} if it is in the image
of this section.

Natural line bundles appear in other geometric constructions as
well. Let \(c:(Y,o)\to (X,o)\) be the universal abelian cover of
\((X,o)\), \(\pi_Y:\widetilde{Y}\to Y\) the normalized pullback of
\(\pi\) by \(c\), and
\(\widetilde{c}:\widetilde{Y}\to\widetilde{X}\) the morphism which
covers \(c\). Then the action of \(H\) on \((Y,o)\) lifts to
\(\widetilde{Y}\) and one has an $H$-eigenspace decomposition
(\cite[(3.7)]{Line} or \cite[(3.5)]{Opg}):
\begin{equation}\label{eq:01}
\widetilde{c}_*\cO_{\widetilde{Y}}=\bigoplus _{l'\in Q}
\cO(-l'),\end{equation}  where $\cO(-l')$ (for $l'\in Q$) is the
$\theta([l'])$-eigenspace of $\widetilde{c}_*\cO_{\widetilde{Y}}$.
This is compatible with the eigenspace decomposition of
$\cO_{Y,o}$ too.


This allows us to regard each $H^0(\widetilde{X},\cO(-l'))$ as the
$\theta([l'])$-eigenspace of the $\setC$-vector space of
meromorphic functions $\gs$ on $\tilde{Y}$ with ${\rm
div}(\gs)\geq \widetilde{c}^*(l')$ (where, in fact,
$\widetilde{c}^*(l')$ is an {\em integral} cycle supported on
$\pi_Y^{-1}(o)$ by \cite[(3.3)]{Line}).


\subsection{Second motivation: equivariant Hilbert series.}\label{SM} For
natural line bundles, the dimensions from (\ref{eq:00})(a) can be
organized in a generating function. Indeed, once a resolution
$\pi$ is fixed, $\cO_{Y,o}$ inherits the {\em divisorial
multi-filtration} (cf. \cite[(4.1.1)]{CDGb}):
\begin{equation}\label{eq:03}
\cF(l'):=\{ f\in \cO_{Y,o}\,|\, {\rm div}(f\circ\pi_Y)\geq \widetilde{c}^*(l')\}.
\end{equation}
Let $\hh(l') $ be the dimension of the $\theta([l'])$-eigenspace
of $\cO_{Y,o}/\cF(l')$. Then, one defines  the {\em equivariant
divisorial Hilbert series} by
\begin{equation}\label{eq:04}
\cH(\bt)=\sum _{l'=\sum l_vE_v\in L'}
\hh(l')t_1^{l_1}\cdots t_s^{l_s}=\sum_{l'\in L'}\hh(l')\bt^{l'}\in
\setZ[[L']].
\end{equation}
(We are aware that in $\cH(\bt)$ many coefficients are repeated
several times, and that $\sum_{l'\in L'_{\geq 0}} \hh(l')\bt^{l'}$
contains the same information and it lives in a ring of power
series instead of a $\setZ[L']$-module as $\cH(\bt)$ does.
However, we prefer the definition (\ref{eq:04}) since it already
appeared in the literature together with the nice relation
(\ref{eq:06}).)

Notice that the terms of the sum reflect the $H$-eigenspace
decomposition too.  E.g., $\sum_{l\in L}\hh(l)\bt^{l}$ corresponds
to the $H$-invariants, hence it is the {\em Hilbert series}  of
$\cO_{X,o}$ associated with the $\pi^{-1}(o)$-divisorial
multi-filtration  (considered and intensively studied, see  e.g.
\cite{CHR} and the citations therein, or \cite{CDG}). The second
central problem is:
 {\em Recover (for some families of
singularities) $\cH(\bt)$ from $\Gamma$}.

Write $l'\in L'$ as $l'=r_{[l']}+l$ for some $l\in L$. Then (see e.g.
\cite[(4.2.4)]{CDGb})
\begin{equation}\label{eq:05}
\hh(l')=\dim
\frac{H^0(\cO(-r_{[l']}))}{H^0(\cO(-r_{[l']}-\max\{0,l\}))}.
\end{equation}
Hence in those cases when the dimensions from (\ref{eq:00})(a)
--- applied for bundles $\cO(-l')$, $l'\in Q$ --- are topological,
$\cH(\bt)$ is topologically determined too.


\subsection{Third motivation: Campillo--Delgado--Gusein-Zade
identity.}\label{TM} Comparison of $\cH$ with topological
invariants lead Campillo, Delgado and Gusein-Zade to consider the
next series, cf. \cite{CDG,CDGEq}:
\begin{equation}\label{eq:06}
\cP(\bt)=-\cH(\bt) \cdot \prod_v(1-t_v^{-1})\in \setZ[[L']].
\end{equation}
Above, $\setZ[[L']]$ is regarded as a module over  $\setZ[L']$.
Campillo, Delgado and Gusein-Zade for several families of
singularities proposed a {\em topological description for
$\cP(\bt)$} (presented here in (\ref{th:21})). For
more positive cases when that description holds, and for some
counterexamples too, see \cite{CDGb}. Thus, it is natural to ask:
{\em how general is this phenomenon (i.e. topological
characterization), and where are its limits}?

Although the multiplication by   $\prod_v(1-t_v^{-1})$ in
$\setZ[[L']]$ is not injective, hence apparently $\cP$ contains
less information then $\cH$ (cf. (\ref{eq:06})), they determine
each other (a fact noticed in \cite{CDG} for the non-equivariant
case, the precise `invertion' follows from (\ref{eq:10b}) below).


\section{General results}\label{s:general}

\subsection{}
The next proposition puts in evidence that in all the
invariants listed in the previous section the crucial ingredient
is:
\begin{equation}
  \label{eq:1}
  h_{\linebundle} \coloneqq \sum_{I \subseteq \cV}
(-1)^{\left\lvert I \right\rvert + 1}
  \dim \smartfrac{H^0 \left( \linebundle \right)}
   {H^0 \left( \linebundle \left( -E_I\right)
  \right)},
\end{equation}
where \(E_I=\sum_{v \in I} E_v\). Notice that  if \((c_1(\linebundle),E_v)
< 0\) for some
$v\in\cV$,  then for any \(I\not\ni v\) one has an isomorphism
\(H^0(\linebundle(-E_{I\cup
v}))\to H^0(\linebundle(-E_I))\), hence
\begin{equation}\label{eq:1b}
h_{\linebundle}=0 \ \ \mbox{whenever \ $-c_1(\linebundle)\not\in\cS'$.}
\end{equation}
Therefore, in the next expressions (\ref{eq:07})-(\ref{eq:08}) the
sums are finite (cf. (\ref{eq:finite})).

In fact, $\cP$ is also  supported on the anti-nef cone $\cS'$ (see
e.g. \cite[(4.2)]{CDGb}), and by (\ref{eq:05}) and (\ref{eq:06})
(for more details see \cite[(4.2.12)]{CDGb}) the generating function
for  $\{h_{\linebundle}:\linebundle \ \mbox{natural}\}$ is exactly
the series $\cP$, sitting in the ring $\setZ[[\cS']]$:
\begin{equation}\label{eq:10}
      \sum_{l'\in L'} h_{\cO(-l')} \bt^{l'} =\cP(\bt).
\end{equation}

\begin{proposition}\label{th:1} Let $(X,o)$ be a normal surface singularity whose link is a
rational homology sphere. Fix one of its good resolutions
$\widetilde{X}\to X$. Denote by $K\in L'$ the {\em canonical
class} satisfying $(K+E_v,E_v)=-2$ for all $v\in\cV$. Then for any
line bundle $\linebundle\in{\rm Pic}(\widetilde{X})$ the following
facts hold:

\begin{enumerate}
\item\label{th:1.1}
For every effective cycle \(l \in  L_{\geq 0}\)
\begin{equation}\label{eq:07}
    \dim \smartfrac{H^0 \left( \linebundle \right)}
    {H^0 \left( \linebundle \left( - l \right) \right)} =
    \sum_{ a\in L_{\geq 0},  a \ngeq l}
    h_{\twistedbundle{\linebundle}{-a}}.
\end{equation}
In particular, for any $l'\in L'$  (via (\ref{eq:05})) one has
\begin{equation}\label{eq:10b}
\hh(l')=\sum_{a\in L,\, a\not\geq 0} h_{\cO(-l'-a)}.
\end{equation}
\item\label{th:1.2}
There exists a constant $\mathrm{const}_{[\linebundle]}$, depending only on the
class of $[\linebundle]\in Pic(\widetilde{X})/L$,  such that
\begin{equation}\label{eq:08}
   - h^1 \left( \linebundle \right) = \sum_{a \in L,\, a \nleq 0}
    h_{\twistedbundle{\linebundle}{a}} +
    \mathrm{const}_{[\linebundle]} +
    \smartfrac{{\left(K - 2 c_1  \left( \linebundle \right)\right)}^2 + |\cV|}
{8}.
  \end{equation}
\end{enumerate}
\end{proposition}
For the proof see (\ref{ss:prth1}). Some remarks are in order.
\begin{remark}\label{re:UJ} \
\begin{enumerate}
\item\label{re:UJ1}
From (\ref{th:1}) follows that  the set
$\{h_{\linebundle}:\linebundle \ \mbox{natural}\}$ determines all
the ranks  $\dim H^0(\cO(l'))/H^0(\cO(l'-l))$. For $h^1(\cO(l'))$
additionally one also needs the constants
$\{\mathrm{const}_{h}\}_{h\in H}$, where
$\mathrm{const}_{[l']}=\mathrm{const}_{[\cO(l')]}$.

\item\label{re:UJ2}
(\ref{eq:08}) identifies in a natural way (in the world of
algebraic geometry) the constants $\{\mathrm{const}_{h}\}_{h\in
H}$. There are two distinguished  regions for the Chern classes
$l'$ of special interest. First, if one takes  in (\ref{eq:08})
$\linebundle=\cO(-r_h+l)$ with $l\in L_{\geq 0}$, then
\begin{equation}\label{eq:sw}
  - h^1 \left( \cO(-r_h+l)\right) =
    \mathrm{const}_{-h} +
    \smartfrac{{\left(K +2 r_h-2l\right)}^2 + |\cV|}{8}.
\end{equation}
Hence, $\mathrm{const}_{-h}+((K+2r_h)^2+|\cV|)/8$ appears as $ - h^1
\left( \cO(-r_h)\right)$, or as the constant term of the
`quadratic function' $l\mapsto  - h^1 \left( \cO(-r_h+l)\right)$.

On the other hand, for $\linebundle =\cO(-l')$ with $l'\in
-K+\cS'$, by the vanishing of $h^1(\linebundle)$  \cite[(3.2)]{Laufer72},
one has
\begin{equation}\label{eq:KV} \mathrm{const}_{[-l']}=
-\sum_{a\in L,\, a\ngeq 0} h_{\cO(-l'-a)}\ -
\frac{(K+2l')^2+|\cV|}{8}.\end{equation}
\end{enumerate}
\end{remark}

Identity (\ref{eq:sw}) can be compared with the one appearing in
the `Seiberg-Witten Invariant Conjecture'
\cite{nemethi02:_seiber_witten,CDGb}. This predicts the
validity of an identity which can be obtained from
(\ref{eq:sw}) by replacing $ \mathrm{const}_{h}$ by the Seiberg-Witten
invariants of the link $M$. We remind that
the Seiberg-Witten invariants are indexed by the
$\mathrm{spin}^c$-structures of $M$. $\mathrm{Spin}^c(M)$ is an
$H$-torsor (with action $(h,\sigma)\mapsto h*\sigma$), and with a
distinguished element, the {\em canonical}
$\mathrm{spin}^c$-structure $\sigma_{\mathrm{can}}$. This is
induced by the complex structure of $X$, but it can be recovered
from the combinatorics of $\Gamma$ too (see
\cite{nemethi02:_seiber_witten}). Hence, the 
Seiberg-Witten Invariant Conjecture 
\cite{nemethi02:_seiber_witten} 
has the following reinterpretation:

\begin{conjecture}\label{SWIC}
For `nice' normal surface singularities with rational homology
sphere link the constants $\{{\rm const}_h\}_h$ are the Seiberg-Witten
invariants of the link:
\begin{equation*}
\mathrm{const}_h=\ssw{h*\sigma_{\mathrm{can}}};
\end{equation*}
\end{conjecture}

\subsection{Proof of Proposition  (\ref{th:1})}\label{ss:prth1}
We will use several times the following rewriting of (\ref{eq:1}).
For any fixed $u\in \cV$, by grouping the subsets of $\cV$ in
pairs $\{I,I\cup u\}_{I\not\ni u}$, we get
\begin{equation}\label{eq:hlam}
h_\linebundle=\sum_{I\not\ni
u}(-1)^{|I|}\frac{H^0(\linebundle(-E_I)}{H^0(\linebundle
(-E_I-E_u))}.
\end{equation}
Notice that part (\ref{th:1.2}) implies part (\ref{th:1.1}) by
Riemann-Roch formula, nevertheless in our proof of (\ref{th:1.2})
we will use part (\ref{th:1.1}) too. Hence we start with the proof
of (\ref{th:1.1}). By comparing (\ref{eq:07}) for $l$ and $l+E_v$
(for some $v\in\cV$) and using induction, it is enough to prove
(\ref{eq:07}) only for $l=E_v$ (and arbitrary $\linebundle$). We
simplify the right-hand side of (\ref{eq:07}) in several steps.

Using (\ref{eq:hlam}) for $u=v$, and (\ref{eq:1b}), the right hand
side of (\ref{eq:07}) for $l=E_v$ is
\begin{equation*}
\sum_a\sum_{I\not\ni
v}(-1)^{|I|}\dim\frac{H^0(\linebundle(-a-E_I))}{H^0(\linebundle(-a-E_I-E_v))},
\end{equation*}
where the first sum runs over $a=\sum_{u\not =v}a_uE_u$, with
$0\leq a_u < k$ for some $k\gg 0$.

By combinatorial cancellation, this is
\begin{equation*}
\sum_{I\not\ni
v}(-1)^{|I|}\dim\frac{H^0(\linebundle(-kE_I))}{H^0(\linebundle(-kE_I-E_v))}.
\end{equation*}
But,  the inclusion $H^0(\linebundle(-kE_I-E_v))\hookrightarrow
H^0(\linebundle(-kE_I))$ is an isomorphism  for $I\not=\emptyset$,
$I\not\ni v$,  and $k\gg 0$. In order to prove this, embed both
groups in $H^0(\linebundle)$. Then notice that
$H^0(\linebundle(-l))\setminus \{0\}=\{\gs\in
H^0(\linebundle)\setminus \{0\}\,:\, \mathrm{div}_\linebundle
(\gs)\geq l\}$ for any $l\in L_{\geq 0}$. On the other hand,
since $(\mathrm{div}_\linebundle(\gs)-c_1(\linebundle),E_u)=0$ for
all $u\in\cV$ and $\gs\not=0$, the compactly supported part
$D(\gs)$ of $\mathrm{div}_\linebundle(\gs)$ satisfies $D(\gs)\in
c_1(\linebundle)+\cS'$. But $\cS'$ has the property  that for
all $c_1\in L'$ there is a $k\gg 0$ such that for all 
$D\in c_1+\cS'$ with  $D\geq kE_I$ one has  $D\geq kE_I+E_v$ too.  
Thus (\ref{eq:07}) for $l=E_v$ follows.

Similarly, for (\ref{th:1})(\ref{th:1.2}), one compares the
sheaves $\linebundle$ and $\linebundle(-l)$ for $l\in L_{\geq 0}$.
The cohomology long exact sequence of
$\linebundle(-l)\hookrightarrow \linebundle \twoheadrightarrow
\linebundle_l$ provides
\begin{equation*}
\dim \frac{H^0(\linebundle)}{H^0(\linebundle(-l))}=
\chi(\linebundle_l)-h^1(\linebundle(-l))+h^1(\linebundle).
\end{equation*}
Then, Riemann-Roch formula and a computation shows that the
expression
\begin{equation*}
h^1(\linebundle)+\sum_{a\nleq
0}h_{\linebundle(a)}+\frac{(K-2c_1(\linebundle))^2+|\cV|}{8}
\end{equation*}
is the same for $\linebundle$ and $\linebundle(-l)$.

\subsection{Remarks on the Seiberg-Witten invariant}\label{re:SW}
In the next section we combine (\ref{eq:08}) and the identity from 
(\ref{th:21}) (valid for splice-quotients) to get (\ref{eq:24}).
Assume that this last identity  is true for some singularity and $l'\in
-K+\cS'$ is in the vanishing zone of $h^1(\cO(-l'))$. Then the
vanishing of the right hand side of (\ref{eq:24}) provides a
completely topological identity connecting the Seiberg-Witten
invariants of the link with a combinatorial data of $\Gamma$. In
particular, we get a {\em new combinatorial formula for the
Seiberg-Witten invariants} for all plumbed 3-manifolds which admit
an analytic structure satisfying 
(\ref{eq:24}); e.g. {\em this is the case for all the graphs of
splice-quotients}.

The point is that this topological identity is true for any
$\Gamma$. Indeed, one has:

\begin{theorem}
Fix an arbitrary resolution graph whose associated plumbed 3-manifold is 
a rational homology sphere. Consider all its combinatorial
invariants  as above, and  define  the coefficients $c_{l'}$ by the expansion
\begin{equation}\label{eq:TOP}
\prod_{v\in\cV}
\big(1-\bt^{E_v^*}\big)^{\delta_v-2}=\sum_{l'\in\cS'}c_{l'}\bt^{l'}\in\setZ[[\cS']].
\end{equation}
Take any representative $l'$ of the class $[l']$ in
$l'\in -K+\cS'$. Then the Seiberg-Witten invariants of the
associated plumbed 3-manifold satisfy:
\begin{equation}\label{eq:SWSW}
  \ssw{[-l']*\sigma_{\mathrm{can}}} +
    \smartfrac{(K + 2 l')^2 + |\cV|}{8}= -\sum_{a \in L, a \ngeq 0}  c_{l'+a},
  \end{equation}
In other words, the expression from the right hand side (i.e.
special truncations of the series (\ref{eq:TOP})) admits a
quadratic generalized Hilbert polynomial, whose free coefficient
is the (renormalized) Seiberg-Witten invariant.
\end{theorem}
The validity of the theorem for splice-quotients (in particular
for rational or star-shaped graphs) follows from (\ref{cor:QS}).
The proof of the general case involves different techniques from
those used in the present article, hence will be published in
another article \cite{NSW}. Here we only sketch the argument.

Notice that the left hand side of the identity satisfies some surgery
identities. Since for `simple' graphs (e.g.
strings, or rational graphs) the identity (\ref{eq:SWSW}) is valid, 
it is enough to
verify that the right hand side of (\ref{eq:SWSW}) satisfies a
similar surgery formula as the left hand side. This (again by the
above discussion, when (\ref{eq:SWSW}) holds)
is true for rational graphs, and one can prove
the general case by induction on the number of `bad' vertices of
the graph  by a similar technique as in \cite{NO2}
($v$ is `bad' if $-E_v^2$ is smaller than the valency
$\delta_v$). Indeed, if
$v$ is a bad vertex, one considers the series of graphs $\Gamma_n$
defined by replacing the self-intersection $-E^2_v$ by
$-E^2_v+n\cdot |\det(\Gamma)|$ ($n\in\setZ_{\geq 0}$). The wished
identity is a rational function in $n$, and it is vanishing (by
the inductive step) for any $n$ sufficiently large. Hence it is
vanishing identically.

\section{The Campillo--Delgado--Gusein-Zade formula for splice-quotients} \label{ss:12}

\subsection{}Splice-quotient singularities were introduced by Neumann
and Wahl in \cite{NWuj2}. From any fixed graph $\Gamma$ (which has
some special properties) one constructs an equisingular family of
singularities whose analytic type is strongly linked with one of
its resolutions whose dual graph is $\Gamma$.

Briefly, the construction of one of the members $(X,o)$ of the
family for a fixed $\Gamma$ is the following.  Let $\cE$ denote
the set of {\em ends}, i.e. the set of those vertices $v$ whose
valency $\delta_v$ is 1. Let $\setC\{z\}$ be the convergent power
series in variables $\{z_i\}_{i\in\cE}$. Then the abelian cover
$(Y,o)$ of $(X,o)$ is a complete intersection in $\setC^{|\cE|}$
with $|\cE|-2$ equations (the so-called `splice diagram
equations'). $H$ acts on $\setC\{z\}$ by $h*z_i=
\theta(E^*_i)(h)z_i$ which induces on $\cO_{Y,o}$ the Galois
action of the abelian cover.

Splice-quotients include the rational singularities (and any
resolution) \cite{Ouac-c},  minimally elliptic singularities (with
resolutions where the support of the minimal elliptic  cycle
 is not proper) \cite{Ouac-c}, and weighted-homogeneous singularities
(with their  minimal good resolutions) \cite{Neu}. For more
details see also \cite{NWuj2} and \cite{Opg}.

\begin{theorem}\label{th:21}
Let $(X,o)$ be a splice-quotient singularity associated with the
fixed graph $\Gamma$. Consider its resolution $\pi$ with graph
$\Gamma$ and the analytic invariants $\cH(\bt)$ and $\cP(\bt)$
associated with $\pi$. Then they can be recovered from $\Gamma$.
More precisely:
\begin{equation}\label{eq:22}
\cP(\bt)=\prod_{v\in\cV} \big(1-\bt^{E_v^*}\big)^{\delta_v-2}.
\end{equation}
(The expression of $\cH(\bt)$ follows from (\ref{eq:22}) and 
(\ref{eq:10b}).)
\end{theorem}

The identity (\ref{eq:22}) is the generalization of the
Campillo--Delgado--Gusein-Zade formula (stated in the terminology
of \cite{CDGb}). This, for rational singularities was proved in
\cite{CDGEq}, for minimally elliptic ones in \cite{CDGb}. This is
a common generalization of them, which includes the weighted
homogeneous singularities as well.

\begin{corollary}\label{cor:QS} Conjecture (\ref{SWIC}) is true
for splice-quotient singularities. Moreover, $h^1(\cO(l'))$ is topological 
for any $l'\in L'$. More precisely, consider the expansion (\ref{eq:TOP}).
Then
\begin{equation}\label{eq:24}
   - h^1 \left( \cO(-l')\right) = \sum_{a \in L, a \ngeq 0}
    c_{l'+a} + \ssw{[-l']*\sigma_{\mathrm{can}}} +
    \smartfrac{(K + 2 l')^2 + |\cV|}{8}.
  \end{equation}
\end{corollary}
Indeed, for the first sentence  combine  (\ref{eq:sw})  with
\cite[(2.2.4)]{BN}, which proves the
Seiberg-Witten Invariant Conjecture for splice-quotients
singularities and for line bundles $\cO(-r_h+l)$ with $l\in
L_{\geq 0}$ (i.e. the analogue of the identity (\ref{eq:sw}) with
$\ssw{-h*\sigma_{{\rm can}}}$ instead of ${\rm const}_{-h}$). 
For the second statement  one needs additionally (\ref{eq:22}) too.

In order to prove (\ref{th:21}) we need some preparations.

The author studied splice-quotient singularities in a joint
project with G. Braun; we even prepared a manuscript about some
new results. Nevertheless, the manuscript was split in two
independent parts: Braun's article \cite{Gabor} contains a new
proof of the End Curve Theorem of Neumann and Wahl,  the remaining
part regarding results in connection with the
Campillo--Delgado--Gusein-Zade type identity constitutes the
present section. Inevitably, some overlaps remained, but they
emphasize different aspects.

\subsection{Preparations. Singularities satisfying the `end curve condition'}\label{ECC}
We fix a singularity $(X,o)$ and one of its good resolutions
$\pi:\widetilde{X}\to X$ with graph $\Gamma$. As above, $\cV$
(resp. $\cE$) denotes the set of vertices (resp. ends) of
$\Gamma$.

  Consider an irreducible curve on $\widetilde{X}$ transversal to $E$
  intersecting exactly one $E_v$  ($v\in\cV$) at one point.
  A \emph{cut function} for this curve is a function $f\in \cO_{X,o}$ such that the
  divisor of $f\circ \pi$  is supported on $E$ and the curve.
  The curve is a \emph{cut of $E_v$} if it has a cut function.
\begin{definition}\label{def:31}\cite{NWuj2}
  The \emph{end curve condition} for $\pi$  requires a cut $H_i$
  for {\em each end component} $E_i,\ i\in\cE$. We call such an $H_i$ {\em end curve} and the
corresponding function {\em end curve function}.
\end{definition}

 Recall that any splice-quotient singularity satisfies
the end curve condition by construction (\cite[(7.2)(6)]{NWuj2}):
the $|H|$-powers of the coordinate functions of $\setC^{|\cE|}$
serve as end curve functions.

\begin{bekezd}\label{rem:31}{\em (End curve sections.)} \cite{NWuj2} For a fixed
$v\in\cV$, let $e_v$ be the order of $[E^*_v]$ in $H$. Since ${\rm
Pic}(\widetilde{X})$ has no torsion, if $H_v$ is a cut of $E_v$,
one can take a cut function  $f_v$ such that ${\rm div}(f_v\circ
\pi)=e_v(H_v+E^*_v)$. In particular, $\cO(-E^*_v)=\cO(H_v)$.
\end{bekezd}
Consider now the universal abelian cover of $(X,o)$,  maps $c$ and
$\widetilde{c}$ as in (\ref{FM}), and  a function $f_v$ for some
$v\in \cV$ as in the previous paragraph. Then $f\circ \pi\circ
\widetilde{c}$ \, is an $e_v$-power of some $z_v\in
H^0(\widetilde{Y},\cO_{\widetilde{Y}})$ ($\simeq\cO_{Y,o}$). Since
$z_v$ is in the $\theta(E^*_v)$-eigenspace, by (\ref{eq:01}) and a
divisor verification  $z_v\in H^0(\widetilde{X}, \cO(-E^*_v))$,
where its divisor is $H_v$.

In this construction of $z_v$ from $f_v$ there is an ambiguity
with multiplication by an $e_v$-root of unity what we will
disregard.

If $\pi:\widetilde{X}\to X$  satisfies the {\em end curve
condition}, the sections  $z_i$ for $i\in \cE$ constructed from (a
fixed set of) end curve functions will be  called {\em end curve
sections}.

Usually, for each $i\in \cE$, one fixes {\em one} cut.
Nevertheless, in order to run the theory properly, for a graph
which has {\em only one vertex}, one fixes {\em two} disjoint
cuts, hence one gets two end curve sections. (For this special
case, the reader is invited to replace some of the indices used
below accordingly.)

\begin{bekezd}\label{ss:mon}{\em (Monomials in
$H^0(\widetilde{X},\cO(-l'))$.)} Fix a subset $\cW\subset \cV$,
and assume that $\pi$ admits cut functions and cuts $H_w$ for any
$w\in \cW$. Fix $l'\in L'$. For any collection $\alpha_w\in
\setZ_{\geq 0}$ ($w\in \cW$), with $[\sum_w\alpha_wE^*_w]=[l']$ in
$H$, consider the monomial $z^\alpha=\prod_{w\in
\cW}z_w^{\alpha_w}$ in $\{z_w\}_w$ as a meromorphic function on
$\widetilde{Y}$. It is in the $\theta([l'])$-eigenspace, hence it
is a meromorphic section of $\cO(-l')$. Since its divisor ${\rm
div}_{\cO(-l')}(z^\alpha)$ is $\sum_w\alpha_w(H_w+E^*_w)-l'$,
$z^\alpha$ is a global holomorphic section of $\cO(-l')$ if and
only if $\sum_w\alpha_wE_w^*-l'\in L_{\geq 0}$.\end{bekezd}
\begin{definition}\label{def:3.1} Assume that $\pi$ satisfies the end
curve condition. Consider the end curve sections
$\{z_i\}_{i\in\cE}$ associated with a fixed set of end curve
functions. An element of $H^0(\widetilde{X},\cO(-l'))$ of the form
$z^\alpha$ with $\cW=\cE$ is called {\em monomial section}.

A global sections $z^\alpha$ of $\cO(-l')$, where
$\{\alpha_v\}_{v\in \cW}$ is indexed by a subset
$\cW\supsetneqq\cE$ (provided that each $\{z_w\}_{w\in\cW}$
exists) is called an {\em extended monomial section}.

To any collection $\{\alpha_i\}_{i\in \cE}$
($\alpha_i\in\setZ_{\geq 0}$)  we associate its {\em monomial
cycle} $D(\alpha):=\sum_{i\in\cE}\alpha_iE^*_i\in L'$.
\end{definition}

\begin{example}\label{ex:31}
For any $l'\in L'$ and $l\in L_{\geq 0}$, the classes of non-zero
monomial sections in $H^0(\widetilde{X},\cO(-l'))/
H^0(\widetilde{X},\cO(-l'-l))$ are indexed by
\begin{equation*}\label{eq:3.1}
\{\alpha\,|\, \alpha_i\geq 0\ \ \mbox{for all $i\in\cE$}; \ a\geq
0, \ a\not\geq l, \ \mbox{where $a=D(\alpha)-l'\in L$}\}.
\end{equation*}
E.g., if $l=E_v$ for some $v\in\cV$ with $(l',E_v)>0$, then this
index set is empty: the existence of such an $\alpha$ would imply
$0>(D(\alpha)-l',E_v)\geq 0$. This chimes in with the fact that
$H^0(\cO(-l'))/ H^0(\cO(-l'-E_v))$ embeds into
$H^0(E_v,\cO(-l'))=0$.
\end{example}

\begin{bekezd}\label{ss:res}{\em (Restrictions.)} Next we introduce the notations
of the  inductive step used in the proof. Let $\pi$ be a fixed
resolution with dual graph $\Gamma$ and $|\cV|>1$.  We fix a
vertex $v\in\cE$; let $E_w$ be the unique irreducible exceptional
curve which intersects $E_v$. Set $\bar{E}=\cup _{u\in\cV\setminus
v}E_u$.  Let $\wxb$ be a sufficiently small neighbourhood of
$\bar{E}$ in $\widetilde{X}$, and $\bar{\Gamma}$ the  dual graph
of $\bar{E}$. We denote by $\bar{E}_u$ ($=E_u$), $u\in\cV\setminus
v$,  the base elements of the new lattice $\bar{L}$, and by
$\bar{E}^*_u$ their $\bar{\Gamma}$-duals in $\bar{L}'$. Notice
that $\bar{E}$ analytically can be contracted (denote this by
$\bar{\pi}$) giving rise to a singularity $(\bar{X},o)$. Set
$\bar{\cE}$ for the ends of $\bar{\Gamma}$ and $\cE'=\cE\setminus
v$.

If $(X,o)$ admits the end curves $\{H_i\}_{i\in\cE}$ in
$\widetilde{X}$ cut out by the end curve functions
$\{f_i\}_{i\in\cE}$,  then   $\wxb$ inherits some compatible cuts
and end curves. Indeed:
\end{bekezd}
\begin{lemma}\label{lem:31} \cite[(2.15)]{Opg}.
The curves $\{H_i\}_{i\in \cE'}$ and $\bar{H}_w=E_v\cap \wxb$ are
cuts of $\bar{E}$ in $\wxb$. In particular, the resolution
$\bar{\pi}$ of $(\bar{X},o)$ satisfies the end curve condition
with distinguished end curve functions inherited from
$\{f_i\}_{i\in\cE}$.
\end{lemma}

\begin{proof}
Let $m_i$ be the vanishing order of $f_i$ along $E_v$. Then for
$\bar{\pi}$,  $f_v|\wxb$ is a cut function for $\bar{H}_w$, and
$\bar{f}_i=f^{m_v}_i/f^{m_i}_v|\wxb$ is a cut function for $H_i$,
$i\in \cE'$.
\end{proof}

Let $i:\bar{L}\hookrightarrow L$ be the natural lattice embedding
$\bar{E}_u\mapsto E_u$, $u\in\cV\setminus v$. Its dual  $R:L'\to
\bar{L}'$ is defined by $(R(l'),\bar{l})=(l',i(\bar{l}))$ (or by
$R(E^*_v)=0$ and $R(E^*_u)=\bar{E}^*_u$ for $u\in \cV\setminus
v$), and one also has
\begin{equation}\label{eq:3.2}
R(E_v)=-\bar{E}^*_w \ \ \mbox{and} \ \ R(E_u)=\bar{E}_u \ \ \mbox{for $u\in\cV\setminus v$}.
\end{equation}

\begin{lemma}\label{lem:32} The restriction of any natural line
bundle to $\wxb$ is natural. In fact, the restriction of
$\cO_{\widetilde{X}}(-l')$ is isomorphic to $\cO_{\wxb}(-R(l'))$.
\end{lemma}

\begin{proof}
We need to show that some power of the restriction has the form
$\cO_{\wxb}(\bar{l})$ for some $\bar{l}\in\bar{L}$. This, by
taking $\cO_{\widetilde{X}}(-l')^{|H|}$, reduces the proof to the
case $\cO_{\widetilde{X}}(E_u)|\wxb$. But this, for $u\not=v$ is
$\cO_{\wxb}(\bar{E}_u)$, and for $u=v$ it is
$\cO_{\wxb}(\bar{H}_w)$, which equals $\cO_{\wxb}(-\bar{E}^*_w)$
by  (\ref{lem:31}) and  (\ref{rem:31}); all of them are natural.
\end{proof}
%
%
%
Here, a word of warning is necessary.  Consider the restriction
map
\begin{equation*}
R^0:H^0(\widetilde{X},\cO_{\widetilde{X}}(-l'))\longrightarrow
H^0(\widetilde{X}(\bar{E}),\cO_{\widetilde{X}}(-l')|_{\widetilde{X}(\bar{E})}).
\end{equation*}
Although $\cO_{\widetilde{X}}(-l')|_{\widetilde{X}(\bar{E})}$ and
$\cO_{\widetilde{X}(\bar{E})}(-R(l'))$ are isomorphic, the
isomorphism connecting them might not send the restriction of a
monomial section of $\cO_{\widetilde{X}}(-l')$ to a monomial
section of $\cO_{\widetilde{X}(\bar{E})}(-R(l'))$.

\subsection{}\label{ss:ktl} In the proof of (\ref{th:21})
we will also use the following lemma. We separate its proof in \S
\,\ref{s:PROOF}.

\begin{lemma}\label{lem:35} \cite{Gabor}
Let $\pi:\widetilde{X}\to X$ be the resolution  which satisfies
the {\em end curve condition}. Let $\linebundle
=\cO_{\widetilde{X}}(-l')$ ($l'\in L'$) be a natural line bundle
on $\widetilde{X}$. Then:
\begin{enumerate}
 \item\label{lem:35.1} For some \(l \in L_{{} \geq 0}\)
    consider the vector space
    \(V:=\smartfrac{H^0(\widetilde{X},
    \linebundle)}{H^0(\widetilde{X},\linebundle( - l ))}\).
    Assume that a subset of $H^0(\widetilde{X},
    \linebundle)$ has the following property:
     for any non-zero class in $V$ of a monomial section of
     $\linebundle$ the set contains an element with the same divisor.
     Then the classes of the elements of the set generate $V$ as a vector space.

  \item\label{lem:35.2} Assume that $|\cV|>1$,  \(v\in\cE\) and
    \(( c_1(\linebundle ),E_v)\geq 0\).
    Set $E_w$ and $\wxb$ as in (\ref{ss:res}). Then the restriction map
    induces an isomorphism
    \begin{equation}
      \label{eq:3.9}R^\flat:
      \frac{H^0( \widetilde{X},\linebundle)}{H^0(\widetilde{X},
      \linebundle( - E_w ))} \longrightarrow
      \frac{H^0( \wxb,\linebundle|_{\wxb})}{H^0(\wxb,\linebundle |_{\wxb}( - E_w ))}.
    \end{equation}
\end{enumerate}
\end{lemma}

\subsection{Proof of (\ref{th:21})}\label{ss:prth2.1}
Notice that by the  `change of variables' $\{x_v=\bt^{E^*_v}\}_v$,
(\ref{th:21}) is equivalent to
\begin{equation}\label{eq:23}
\sum _{k_v\geq 0} h_{\cO(-\sum_v k_v E_v^*)}\prod_{v\in \cV}\,
x_v^{k_v}=\prod_{v\in \cV} (1-x_v)^{\delta_v-2}.
\end{equation}
Next we prove (\ref{eq:23}) for all  singularities which satisfies
the end curve condition. Since splice-quotients satisfies this the
result follows.

The proof is by induction on the number of vertices (where
we will also use that  the restriction preserves the class of
singularities satisfying the end curve condition, and also their
natural line bundles, cf. (\ref{lem:31}) and (\ref{lem:32})).  The
statement is clear for graphs having only one vertex \(v\), since
it is just a reformulation of \(\dim \smartfrac{H^0 \left(
\bundlefromdivisor{- k
    E_v^*} \right)}{H^0 \left( \bundlefromdivisor{- k
    E_v^* - E_v} \right)} =h^0(\cO_{\setP^1}(k))= k + 1\) for all  \(k\geq 0\).

For a graph \(\Gamma\) with $|\cV|>1$, let $v$, $w$,
$\bar{\Gamma}$ be as in (\ref{ss:res}), and denote the restriction
of $\cB=\cO(-\sum_uk_uE^*_u)\in {\rm Pic}(\widetilde{X})$ to
$\wxb$ by $\bar{\cB}$. Then consider (\ref{eq:hlam}) for $u=w$.
Since $(c_1(\cB),E_v)=k_v\geq 0$ and $-(E_I,E_v)\geq 0$,
(\ref{lem:35})(\ref{lem:35.2}) can be applied for each
$\linebundle=\cB(-E_I)$. Moreover, separating the cases when $I$
contains $v$ or not, we get (cf. (\ref{eq:3.2}))
\begin{equation*}
  h_{\cB} =
h_{\bar{\cB}}   - h_{\overline{\twistedbundle{\cB}{- E_v}}}=
h_{\bar{\cB}}   - h_{\bar{\cB}( \bar{E}^*_w)}.
\end{equation*}
Notice that, by (\ref{eq:1b}), $h_{\bar{\cB}(\bar{E}^*_w)}=0$
whenever $k_w=0$. Therefore, if $\cP_{X,o}$ denotes the left hand
side of (\ref{eq:23}), we have
\begin{equation*}
  \cP_{X,o} = \cP_{\bar{X},o} \cdot \sum_{k_v\geq 0} x_v^{k_v} \cdot (1-x_w).
\end{equation*}
But the right hand side of (\ref{eq:23}) satisfies the same
inductive formula.

\subsection{Remarks.} (a)
For splice-quotients, besides the divisorial multi-filtration,
there exists another filtration  on $\cO_{Y,o}$, indexed by the
same index set $L'$. The {\em monomial filtration} on $\setC\{z\}$
is defined  by the degrees of monomials:
\begin{equation}\label{eq:21}
\cG(l'):=\big\{\ \sum_\alpha a_\alpha z^\alpha\in\setC\{z\}\,:\,
\sum_{i\in\cE}\alpha_iE_i^*\geq l'\ \mbox{ whenever
$a_\alpha\not=0$}\big\}.
\end{equation}
Let $\phi:\setC\{z\}\to \cO_{Y,o}$ send the variable $z_i$ to the
corresponding end curve section (denoted by $z_i$ too, cf.
(\ref{rem:31})). By the definitions  $\phi(\cG(l'))\subset
\cF(l')$. But, in fact, (\ref{th:21}) implies their equality.
Indeed, let $\cH_\cG$, respectively $\cP_\cG$, be the series
associated with $\phi(\cG)$ similarly as $\cH$ and $\cP$ are
associated with the filtration $\cF$, cf. (\ref{eq:04}) and
(\ref{eq:06}). One shows by a similar argument as in
(\ref{ss:prth1}) that $\cH_\cG$ can be recovered from $\cP_\cG$ by
a similar way as $\cH$ from $\cP$ (cf. (\ref{eq:10b})). Since the
abelian cover $(Y,o)$ is a complete intersection with $\delta_v-2$
equations for each $v$ with $\delta_v\geq 3$, the series $\cP_\cG$
equals the right hand side of (\ref{eq:22}). But $\cP=\cP_\cG$
implies $\cH=\cH_\cG$. This together with $\phi(\cG(l'))\subset
\cF(l')$ implies $\phi(\cG)=\cF$.

The same fact follows by a standard argument from
(\ref{lem:35})(\ref{lem:35.1}) too, cf. \cite{Gabor}, since it
implies that the completion
$\widehat{\phi}:\widehat{\cG(l')}\to\widehat{\cF(l')}$ is onto,
hence  $\phi(\cG(l'))=\cF(l')$ too.

This generalizes a result of Okuma \cite[(3.3)]{Opg}, where the
filtrations  are associated with only one  node.

(b) For any fixed vertex $u$ of $\Gamma$ consider the divisorial
filtration $\cF_u$ associated with the irreducible component
$E_u$, and let $\cP_u$ be the Poincar\'e series of
$Gr_{\cF_u}\cO_{Y,o}$. Then, by \cite{CDGb}, $Gr_{\cF_u}\cO_{Y,o}$
is obtained from $\cP(\bt)$ by substitutions $t_w=1$ for all
$w\in\cV\setminus u$. Hence, by taking $H$-invariants (for details
see e.g. \cite{CDGb}), (\ref{th:21}) implies that the Poincar\'e
series of $Gr_{\cF_u}\cO_{X,o}$ is
\begin{equation}\label{zeta2}
 \frac{1}{\lvert H \rvert} \cdot \sum_{\rho \in \widehat{H}}\,
 \prod_{v\in\cV} {(1-\rho([E^*_v])
t_u^{-(E^*_v,E^*_u)})}^{\delta_v-2}.
\end{equation}
This provides topologically the semigroup of the $E_u$-valuation
too.

\section{Principal cycles and the multiplicity}\label{sec:princ}

\subsection{Principal $\setQ$-cycles.}\label{ss:pqd}
 Let $(X,o)$ be a normal surface singularity and
$\pi:\widetilde{X}\to X$ one of its fixed resolutions. We use the
notations of (\ref{ss:11}) and also define
\begin{equation*}
\cS:=\cS'\cap L=\{l\in L\,:\, (l,E_v)\leq 0 \ \ \mbox{for all
$v\in \cV$}\}.
\end{equation*}
Let $\mathrm{supp}(l)\subset E$ denote the support of a cycle
$l\in L$. For any set $\{l'_j\}_{j\in\cJ}$ with
$l'_j=\sum_vl'_{j,v}E_v\in L'$ set $\inf_jl'_j:=\sum_vl'_vE_v\in
L\otimes \setQ$, where $l'_v=\min_jl'_{j,v}$.
\begin{definition}\label{def:61} A rational cycle $l'\in L'$  is
called a  {\em principal $\setQ$-cycle} if $\cO(-l')$ has  a global
holomorphic section $\gs$ which is not zero on  any of the
exceptional components. Their set will be denoted by $\cPr'$. The set
$\cPr$ of {\em principal cycles} is defined as $\cPr'\cap L$; it
consists of precisely the restrictions to $E$ of the divisors of
$\pi$-pullbacks of analytic functions from $\cO_{X,o}$.
\end{definition}

Clearly,  $\cPr'$ (respectively $\cPr$) is a sub-semigroup of
$\cS'$ (resp. of $\cS$).

In general, for an arbitrary singularity, it is very difficult
(unsolved) task to decide if an element of $\cS'$ is principal or
not. E.g., if $(X,o)$ is rational (and $\pi$ is arbitrary), or
minimally elliptic (and $\pi$ is minimal and good) then
$\cPr=\cS$; but, in general, $\cPr\not=\cS$. The point is that, in
general, the semigroup $\cPr$ cannot be characterized
topologically, it strongly depends on the analytic structure
supported by the topological type fixed by $\Gamma$.

The next result characterizes $\cPr'$ combinatorially in
$L'(\Gamma)$, for a splice-quotient singularity associated with
$\Gamma$. Recall that $\cE$ denotes the set of end-vertices of
$\Gamma$.

\begin{theorem}[Characterization of principal $\setQ$-cycles]\label{th:61}
Let $(X,o)$ be a splice-quotient singularity associated with the
graph $\Gamma$, and consider its resolution $\pi$ with dual graph
$\Gamma$. Then for any $l'\in\cS'$ the following facts are
equivalent:

\begin{enumerate}\label{EN:1}
\item\label{EN:1.1}
$l'\in\cPr'$;
\item\label{EN:1.2}
for each $v\in \cV$ there exists an effective cycle $l_v\in
L_{\geq 0}$ such that:
\begin{enumerate}\label{en:1}
\item\label{en:1.1}
$E_v\not\subset \mathrm{supp}(l_v)$,
\item\label{en:1.2}
$(l_v,E_u)=-(l',E_u)$ for any $u\not\in\cE$,
\item\label{en:1.3}
$(l_v,E_u)\leq -(l',E_u)$ for any $u\in\cE$;
\end{enumerate}
\item\label{EN:1.3}
there exists finitely many monomial cycles
$\{D(\alpha_{(k)})\}_k\in l'+L$ so that
$l'=\inf_kD(\alpha_{(k)})$.
\end{enumerate}
\end{theorem}
\begin{proof} (\ref{EN:1.1})$\Rightarrow$(\ref{EN:1.2}).
If $l'\in \cPr'$ then the quotient
$H^0(\cO(-l'))/H^0(\cO(-l'-E_v))\not=0$ for each $v\in\cV$.
Then, by (\ref{lem:35}) (see also (\ref{ex:31})), there exists a
monomial $z^{\alpha_{(v)}}$ so that
\begin{equation}\label{eq:61}
\sum_{i\in\cE}\alpha_{(v),i}\,E_i^*-l'=l_v,
\end{equation}
where $l_v\in L_{\geq 0}$ and $E_v\not\subset \mathrm{supp}(l_v)$.
Hence (\ref{EN:1.2}) follows. In order to prove
(\ref{EN:1.2})$\Rightarrow$(\ref{EN:1.3}), define the non-negative
integers $\alpha_{(v),i}:=-(l'+l_v,E_i)$ for all $i\in\cE$. This
means that (\ref{eq:61}) is satisfied. Since $\inf_v
D(\alpha_{(v)})=\inf_v(l'+l_v)=l'$, (\ref{EN:1.3}) follows too.
For  (\ref{EN:1.3})$\Rightarrow$(\ref{EN:1.1}) notice that if
$l'=\inf_kD(\alpha_{(k)})$, then a linear combination of the
monomials $z^{\alpha_{(k)}}$ works for $\gs$.
\end{proof}

\begin{corollary}\label{cor:62} Let $(X,o)$ be as in
(\ref{th:61}). Then $m\cdot \cS'\subset \cPr$ for some $m\in
\setZ_{>0}$. In fact, if $\delta_v\not=2$ then  $e_vE^*_v\in\cPr$,
where $e_v$ is the order of $[E^*_v]$ in $H$.
\end{corollary}
\begin{proof}
For $v\in\cE$, $e_v E^*_v\in\cPr$ by the definition of
splice-quotients (see e.g. (\ref{rem:31})). 
Next, fix a node  $v$. Let $\{C_k\}_{1\leq k\leq
\delta_v}$ be its branches. By the monomial condition there exists
a monomial cycle $D(\alpha_{k})$ so that $D(\alpha_{k})-E^*_v$ is
effective, integral and supported on $C_k$. Hence
$e_vE^*_v=\inf_kD(e_v\alpha_{k})$.

More generally,  take a vertex $v$ with $\delta_v\geq 2$ adjacent
vertices $\{w_k\}_k$ sitting in the branches $\{C_k\}_k$ of $v$.
Regard $C_k$ as a subgraph, let $\bar{E}^*_i$ be the $C_k$-dual of
$E_i$ in this subgraph, and in $C_k$ consider an arbitrary
integral monomial cycle of type $\sum_{i\in
\cE_{C_k}}\alpha_{k,i}\bar{E}^*_i$. Let $m_{k}$ be the
multiplicity of $E_{w_k}$ in this cycle. This means that
$\sum_{i\in\cE_{C_k}}\alpha_{k,i}E^*_i-m_kE^*_v$ is effective,
integral and supported on $C_k$. Hence $m_v=\mathrm{lcm}_k\{m_k\}$
works because $m_vE^*_v=\inf_kD(\alpha_km_v/m_k)$.
\end{proof}


\begin{remark}\label{rem:61}
(a) It can happen that $e_vE^*_v\not\in\cPr$; e.g. in the case of
the $-13$-vertex of the right graph from (\ref{ex:61})  $e_v=1$
but $E^*_v\not\in\cPr$.

(b) One of the motivations for the characterization of $\cPr$ is
the Nash Conjecture. One way to separate the components of the arc
space is to use the ratios $l_i/l_j$ of the coefficients of the
principal cycles $\sum_il_iE_i\in \cPr$, see e.g. \cite{CaPl}. By
(\ref{cor:62}), for splice-quotients, the set of these ratios
remain the same if we replace $\cPr$ by $\cS$.
\end{remark}

\subsection{The minimal and maximal cycle}\label{ss:MinMax}
Let us recall the following definitions:

\begin{definition}\label{def:62}
Let $Z_{{\rm min}}\in \cS$ be the {\em minimal (or fundamental) cycle}
of $\Gamma$, i.e. the unique minimal non-zero element of $\cS$.
Let $\cO_{\widetilde{X}}(-Z_{{\rm max}})$ be the divisorial part
of the pullback $\pi^*(m_{X,o})$ of the maximal ideal.
Equivalently, $Z_{{\rm max}}$ is the unique minimal non-zero
element of $\cPr$. It is called the {\em maximal cycle} of
$(X,o)$.
\end{definition}
$Z_{{\rm min}}$ can be determined by a combinatorial algorithm of
Laufer \cite[(4.1)]{Laufer72}. Clearly, $Z_{{\rm min}}\leq Z_{{\rm
max}}$, but in general $Z_{{\rm max}}$ cannot be read from
$\Gamma$.

We denote the (finitely generated) semigroup
of integral monomial cycles by 
$$L_\cE:=\{\sum_{i\in\cE}\alpha_iE^*_i\,:\, \alpha_i\in\setZ_{\geq
0}\}\cap L.$$

\begin{corollary}\label{cor:61} Let $(X,o)$ be as in
(\ref{th:61}). Then  $Z_{{\rm max}}=\inf( L_\cE\setminus \{0\})$.
\end{corollary}
\begin{example}\label{ex:61}
Consider the following two graphs, both with $L'=L$.

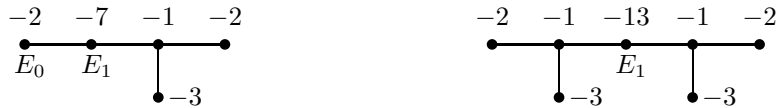
\begin{picture}(400,45)(120,0)
\put(150,25){\circle*{4}}
\put(175,25){\circle*{4}} \put(200,25){\circle*{4}}
\put(225,25){\circle*{4}}
\put(200,5){\circle*{4}} \put(150,25){\line(1,0){75}}
\put(200,25){\line(0,-1){20}}
\put(150,35){\makebox(0,0){$-2$}}
\put(175,35){\makebox(0,0){$-7$}}
\put(177,17){\makebox(0,0){$E_1$}}
\put(152,17){\makebox(0,0){$E_0$}}
\put(200,35){\makebox(0,0){$-1$}}
\put(225,35){\makebox(0,0){$-2$}} 
\put(210,5){\makebox(0,0){$-3$}}

\put(325,25){\circle*{4}} \put(350,25){\circle*{4}}
\put(375,25){\circle*{4}} \put(400,25){\circle*{4}}
\put(425,25){\circle*{4}} \put(350,5){\circle*{4}}
\put(400,5){\circle*{4}} \put(325,25){\line(1,0){100}}
\put(350,25){\line(0,-1){20}} 
\put(400,25){\line(0,-1){20}} \put(325,35){\makebox(0,0){$-2$}}
\put(350,35){\makebox(0,0){$-1$}}
\put(375,35){\makebox(0,0){$-13$}}\put(377,17){\makebox(0,0){$E_1$}}
\put(400,35){\makebox(0,0){$-1$}}
\put(425,35){\makebox(0,0){$-2$}} \put(360,5){\makebox(0,0){$-3$}}
\put(410,5){\makebox(0,0){$-3$}}
\end{picture}

For the left graph $Z_{{\rm min}}\in\cPr$, nevertheless
$\cPr\not=\cS$. In fact, $\cPr\setminus \cS=\{E_1^*\}$, where
$E_1$ is the $-7$-curve. The right graph is an example with
$Z_{{\rm min}}\not\in\cPr$ and with infinite $\cPr\setminus \cS$.
Indeed,  let $E_1$ denote the $-13$-curve and let $E_i$ be any
exceptional curve with valency 1. 
Then  $E_1^*+kE_i^*\not \in\cPr$ for every $k\geq 0$.
In this case  $Z_{{\rm max}}=2Z_{{\rm min}}$.
%
\end{example}

\subsection{Base points and multiplicity.} In general,  $\pi^*m_{X,o}$
has the form $\cO_{\widetilde{X}}(-Z_{{\rm max}})\otimes
\bigotimes_{P\in\cB}\cI_P$, where $\cB$ denotes the finite set of
base-points and $\cI_P$ is an $m_P$-primary ideal for $P\in \cB$.
Next we provide a combinatorial description of the ideals $\cI_P$.

Assume that $\pi$ satisfies the end curve condition, we fix some
end curves $H_i$ for $i\in\cE$ as in (\ref{ECC}). Consider
$D:=E\cup(\cup_{i\in\cE}H_i)$, and let $\mathrm{Sing}(D)$ denote
the singular (double) points of $D$. Fix such a point
$P\in\mathrm{Sing}(D)$ and order the two components of $D$
containing it. Associate with each $\alpha\in L_\cE$ the
multiplicity orders $(a_\alpha,b_\alpha)$ in $\sum_{i\in
\cE}\alpha_i(H_i+E_i^*)-Z_{{\rm max}}$  of the two components of
$D$ containing $P$.

The results (\ref{th:61}) and (\ref{cor:61}) have the following
consequence too:
\begin{corollary}\label{lem:610}
\begin{enumerate}
\item\label{lem:61}
$\cB\subset {\rm Sing}(D)$.  In particular, if $H_i\cap E\in \cB$
for some $i\in \cE$, then the intersection point $H_i\cap E$ is
independent of the choice of $H_i$.
\item\label{lem:62} Fix $P\in{\rm Sing}(D)$.
In some local coordinates $(t,s)$ of $P$, $\cI_P$ is the monomial
ideal generated by $\{t^{a_\alpha}s^{b_\alpha}\}_{\alpha\in
L_\cE}$.
\end{enumerate}
\end{corollary}

\begin{example}
On the left graph of (\ref{ex:61}), $\cB$ consists of only one
point, the intersection of a cut with $E_0$,  the unique curve
$E_i$ with $(Z_{{\rm min}},E_i)<0$. In the case of the right graph
$\cB=\emptyset$. In the next example (\ref{ex:-8}), $\cB$ consists
of one point, which is the intersection of two irreducible
exceptional components.
\end{example}

For every  $P\in \cB$, the pairs  $(a_\alpha,b_\alpha)_{\alpha}$
determine a (convenient) Newton diagram:
\begin{equation*}
N^-_P=\setR_{\geq 0}^2\setminus \mbox{convex
closure}\big\{\,\bigcup_{\alpha\in L_\cE}\,
\big((a_\alpha,b_\alpha)+\setR_{\geq 0}^2\big)\,\big\}.
\end{equation*}
\begin{theorem}\label{th:62} The multiplicity ${\rm mult}(X,o)$ of $(X,o)$ is
topological:
\begin{equation*}
{\rm mult}(X,o)=-Z_{{\rm max}}^2+2\cdot \sum _{P\in\cB} \, {\rm
area}(N^-_P).
\end{equation*}
\end{theorem}
\begin{proof}
For each $P$ let $\phi_P$ and $\psi_P$ be two plane curve
singularities with Newton diagram $N^-_P$ and generic
coefficients. Then (see e.g. the proof of \cite[(2.7)]{Wa})
\begin{equation*}
{\rm mult}(X,o)=-Z_{{\rm max}}^2+\sum _{P\in\cB}
(\phi_P,\psi_P)_P,
\end{equation*}
where $(-,-)_P$ denotes the intersection multiplicity at $P$. One
the other hand it is well-known that
 $(\phi_P,\psi_P)_P=2\cdot {\rm area}(N^-_P)$, see e.g. \cite{Tei},
or \cite[page 276]{Ploski}.
\end{proof}

\begin{example}\label{ex:-8}
Consider the following graph.

\begin{picture}(300,45)(250,0)
\put(325,25){\circle*{4}} \put(350,25){\circle*{4}}
\put(450,25){\circle*{4}}
\put(375,25){\circle*{4}} \put(400,25){\circle*{4}}
\put(425,25){\circle*{4}} \put(350,5){\circle*{4}}
\put(425,5){\circle*{4}} \put(325,25){\line(1,0){125}}
\put(350,25){\line(0,-1){20}}
\put(425,25){\line(0,-1){20}}
\put(325,35){\makebox(0,0){$-2$}}
\put(350,35){\makebox(0,0){$-1$}}
\put(375,35){\makebox(0,0){$-8$}}
\put(400,35){\makebox(0,0){$-8$}}
\put(425,35){\makebox(0,0){$-1$}} \put(360,5){\makebox(0,0){$-3$}}
\put(412,5){\makebox(0,0){$-3$}} \put(450,35){\makebox(0,0){$-2$}}

\put(325,17){\makebox(0,0){$z_1$}}
\put(340,5){\makebox(0,0){$z_2$}}
\put(451,17){\makebox(0,0){$z_3$}}
\put(436,5){\makebox(0,0){$z_4$}}
\end{picture}

In this case $Z_{{\rm max}}=Z_{{\rm min}}$ with $Z_{{\rm max}}^2=-2$. Let $P$
be the intersection point of the two $(-8)$-curves. Then $\pi^*(m_{X,o})=
\cO_{\widetilde{X}}(-Z_{{\rm max}})\otimes m_P$, hence ${\rm mult}(X,o)=3$.

In fact, in this case we can even write the equations of $(X,o)$.
Indeed, if $\{z_i\}_{1\leq i\leq 4}$ are the end curve sections
(as indicated in the picture), then a possible choice for the
splice diagram equations is
$z_1^2+z_2^3+z_3^5=z_3^2+z_4^3+z_1^5=0$.  $H\simeq \setZ_3$ is
generated by $[E^*_2]$, and acts by
$[E^*_2]*(z_1,z_2,z_3,z_4)=(z_1,\bar{\psi} z_2,z_3,\psi z_4)$,
where $\psi=e^{2\pi i/3}$. The generators of the invariants are
$z_1,z_3$, $a=z_2^3$, $b=z_4^3$ and $c=z_2z_4$, hence the
equations of $(X,o)$ are $z_1^2+a+z_3^5=z_3^2+b+z_1^5=ab-c^3=0$,
or $c^3=(z_1^2+z_3^5)(z_3^2+z_1^5)$.
\end{example}

\section{The proof of (\ref{lem:35}).}\label{s:PROOF}
We prove both statements by a simultaneous induction on the number
of vertices of the resolution graph. For this we will use the
notations (\ref{ss:res}) and (\ref{lem:31}). The cut functions
$\{\bar{f}\}_{i\in \cE'}$ and $f_v|\wxb$ induce sections in
$H^0(\wxb,\cO_{\wxb}(-\bar{E}^*_i))$, resp. in
$H^0(\wxb,\cO_{\wxb}(-\bar{E}^*_w))$,  denoted by $\bar{z}_i$ for
$i\in \cE'$, resp. by  $y_w$ (cf. (\ref{rem:31})). Clearly, if
$\delta_w>2$, then $w\not\in \bar{\cE}$, hence $\bar{\cE}=\cE'$.
Hence the new monomial sections have the form $\prod_{i\in \cE'}
\bar{z}_i^{\alpha_i}$ ($\alpha_i\geq 0$); a product of type
$\prod_{i\in \cE'} \bar{z}_i^{\alpha_i}\cdot y_w^\beta$
($\beta\geq 0$) is an extended monomial section. If $\delta_w=2$,
then $\bar{\cE}=\cE'\cup w$, hence all $\{\bar{H}_i\}_i$ and
$\bar{H}_w$ are end curves of $\wxb$, and the new monomials
sections have the form $\prod_{i\in \cE'}
\bar{z}_i^{\alpha_i}\cdot y_w^\beta$ ($\alpha_i\geq 0,\ \beta\geq
0$). If $\delta_w=1$ then for $\wxb$ we preserve both (old and
new) cuts $H_w$ and $\bar{H}_w$, and the new  monomials sections
have the form $\bar{z}_w^\alpha y_w^\beta$.

\vspace{2mm}

 The induction starts with a
one-vertex graph with two end curves $H_1$ and $H_2$, in which
case we have to check only (\ref{lem:35.1}). In this case the
picture is very explicit (and can be verified easily by the
reader):

Set $p=-E^{2}$. Then $E^*=(1/p)E$ and $H=\setZ_p$. The abelian
cover  $Y$ is $\setC^2$, $\cO_{Y,o}=\setC\{z_1,z_2\}$, where $z_1$
and $z_2$ are the two end curve sections (see e.g.
\cite[III.5]{BPV}). Moreover, $\widetilde{Y}$ is the blow up of
$Y$ at $o$ with exceptional curve $\widetilde{E}$, and
$\widetilde{c}^*(E)=p\widetilde{E}$. Set
$\linebundle=\cO_{\widetilde{X}}(-kE^*)$ for some $k\in\setZ$.
Then $H^0(\widetilde{X},\linebundle)
= H^0(\widetilde{Y}, \cO(-k\widetilde{E}))_{\theta([l'])}$ which
equals the subspace of $\setC\{z_1,z_2\}$ generated by monomials
of degree $\deg\geq k$ and $\deg \equiv k\ ({\rm mod} \ p)$.
Therefore, if $l=nE$,  the monomial sections $z^\alpha$ of
$\linebundle$  with $\alpha_1+\alpha_2=k+ip$ ($0\leq i<n$) are
linearly independent and the space spanned by them projects
bijectively onto $V$. Since ${\mathrm
div}_{\linebundle}(z^\alpha)=\alpha_1H_1+\alpha_2H_2+iE$, the
vanishing orders of all these divisors at the intersection point
$H_1\cap E$ are all distinct. Hence, any set of section of
$\linebundle$ with the same cardinality and divisors projects into
a linearly independent set in $V$, which necessarily form a basis
of $V$.

 Now, we consider a resolution $\pi:\widetilde{X}\to X$
whose graph has $|\cV|>1$ vertices. We fix a vertex $v\in\cE$, and
we will use the notations of (\ref{ss:res}). By induction, we
assume that the statements of the lemma are true for $\bar{\pi}$.

First we prove statement (\ref{lem:35.2}) for $\pi$.
Since $R^\flat$ is injective, we have to show its surjectivity.
For this consider all the monomial sections of
$\linebundle|_{\wxb}$ with non-zero class in the target of
$R^\flat$. They have the form
$\bar{M}=\prod_{i\in\cE'}\bar{z}_i^{\alpha_i}\cdot y_w^\beta$
(where $\beta=0$ if $w\not\in \bar{\cE}$) such that
\begin{equation*}\label{eq:3.10}
\sum_{i\in\cE'}
\alpha_i\bar{E}_i^*+\beta\bar{E}_w^*+R(c_1(\linebundle))=\bar{l}
\end{equation*}
for some $\bar{l}\in \bar{L}_{\geq 0}$ and $\bar{l}\not\geq
\bar{E}_w$. In particular, $\bar{l}$ is supported on the closure
of $\bar{E}\setminus \bar{E}_w$, hence  $(i(\bar{l}),E_v)=0$.
Since
$R(\sum_{i\in\cE'} \alpha_iE^*_i-\beta
E_v+c_1(\linebundle)-i(\bar{l}))=0$,
one has
\begin{equation}\label{eq:3.12}
\sum_{i\in\cE'} \alpha_iE^*_i-\beta
E_v+c_1(\linebundle)-i(\bar{l})=-\alpha_v E_v^*
\end{equation}
for some $\alpha_v\in \setZ$. Since  $\beta E_v+i(\bar{l})\in
L_{\geq 0}$, and $\alpha_v=-\beta E^2_v+(c_1(\linebundle),E_v)\geq
0$, we get that
$z^\alpha=\prod_{i\in\cE}z_i^{\alpha_i}$ is a monomial section of
$\linebundle$, such that the divisors of $R^0(z^\alpha)$ and
$\bar{M}$ coincide. By the inductive step the collection
$\{R^0(z^\alpha)\}_\alpha$ is a generator set of the target of
$R^\flat$. In particular, $R^\flat$ is onto proving part
(\ref{lem:35.2}).

Moreover, if we replace above each section $z^\alpha$ by another
section $\gs_\alpha$ of $\linebundle$ with the same divisor, then
by the same inductive argument as above the set
$\{R^0(\gs_\alpha)\}_\alpha$ form a basis of the target of
$R^\flat$. But, since  $R^\flat$ is an isomorphism, the classes of
$s_\alpha$ necessarily form a basis too. This proves part
(\ref{lem:35.1}) for $l=E_w$ and any $\linebundle$ with
$(c_1(\linebundle),E_v)\geq 0$.

Now we turn to statement (\ref{lem:35.1}) for general
$\linebundle$ and $l\in L_{\geq 0}$. First we consider the
following infinite `computation sequence' $\{x_n\}_{n\geq 0}$,
$x_n\in L_{\geq 0}$, such that $x_0=0$, $x_{n+1}=x_n+E_{u(n)}$,
where $u(n)\in\cV$ is provided by the following principle:

\begin{enumerate}[\ \ \ \ (a)]
\item\label{it.1}
If $(-c_1(\linebundle)+x_n, E_u)\leq 0$ for all $u\in\cV$ then
$u(n)=w$.
\item\label{it.2}
If $(-c_1(\linebundle)+x_n,E_u)>0$ for some $u$, then take one of
them for $u(n)$.
\end{enumerate}
One can show (see e.g. \cite[(4.2)]{Line}) that both steps occur
infinitely many times.  Moreover, at the beginning of step
(\ref{it.1}), $-c_1(\linebundle)+x_n\in\cS'$. Since
$\cS'\cap\{l'\in L'\,:\, l'\not\geq l-c_1(\linebundle)\}$ is
finite, cf. (\ref{eq:finite}), we get that $x_k\geq l$ for some
$k$.

Next we analyze the steps of $\{x_n\}_{n=0}^k$ and we show by
induction on $n$ the validity of (\ref{lem:35.1}) for
$V_n:=H^0(\linebundle)/H^0(\linebundle(-x_n))$. In case
(\ref{it.2}), the natural projection $V_{n+1}\to
V_n$ is an isomorphism and the divisors of the monomials can also
be identified  (cf. (\ref{ex:31})).
In case (\ref{it.1}), one has the exact sequence
\begin{equation*}
0\longrightarrow
\frac{H^0(\linebundle(-x_n))}{H^0(\linebundle(-x_n-E_w))}\longrightarrow
V_{n+1}\longrightarrow
V_n\longrightarrow 0.
\end{equation*}
(\ref{lem:35.1}) is valid for $V_n$ by the inductive hypothesis,
and for the left hand side too by the particular case already
proved ($l=E_w$); hence it works for the middle term as well. This
ends the induction showing (\ref{lem:35.1}) for $V_k$.  Finally,
consider the projection $V_k\to
V=H^0(\linebundle)/H^0(\linebundle(-l))$. Any set of
$H^0(\linebundle)$ which satisfy the assumption of
(\ref{lem:35.1}) for $V$ can be completed (by adding monomial
sections whose classes in $V$ are zero) to a set which satisfies
the assumption of (\ref{lem:35.1}) for $V_k$. Hence, since $V_k$
satisfies (\ref{lem:35.1}), so does $V$.

\bibliographystyle{amsplain}\bibliography{hivatkozas}

\makeatletter \def\@strippedMR{}
  \def\@scanforMR#1#2#3\endscan{\ifx#1M\ifx#2R\def
  \@strippedMR{#3}\else\def\@strippedMR{#1#2#3}\fi\fi}
  \renewcommand\MR[1]{\relax\ifhmode\unskip\spacefactor3000 \space\fi
  \@scanforMR#1\endscan MR\MRhref{\@strippedMR}{\@strippedMR}} \makeatother
  \def\cprime{$'$}
\providecommand{\bysame}{\leavevmode\hbox to3em{\hrulefill}\thinspace}
\providecommand{\MR}{\relax\ifhmode\unskip\space\fi MR }
\providecommand{\MRhref}[2]{%
  \href{http://www.ams.org/mathscinet-getitem?mr=#1}{#2}
}
\providecommand{\href}[2]{#2}
\begin{thebibliography}{10}

\bibitem{BPV}
W.~Barth, C.~Peters, and A.~Van~de Ven, \emph{Compact complex surfaces},
  Ergebnisse der Mathematik und ihrer Grenzgebiete (3) [Results in Mathematics
  and Related Areas (3)], vol.~4, Springer-Verlag, Berlin, 1984. \MR{MR749574
  (86c:32026)}

\bibitem{Gabor}
G.~Braun, \emph{Manuscript in preparation}.

\bibitem{BN}
G.~Braun and A.~N{\'e}methi, \emph{{S}urgery formula for {S}eiberg-{W}itten
  invariants of negative definite plumbed 3-manifolds}, \url{arXiv:0704.3145}.

\bibitem{CDG}
A.~Campillo, F.~Delgado, and S.~M. Gusein-Zade, \emph{Poincar\'e series of a
  rational surface singularity}, Invent. Math. \textbf{155} (2004), no.~1,
  41--53. \MR{MR2025300 (2004j:14042)}

\bibitem{CDGEq}
\bysame, \emph{{U}niversal abelian covers of rational surface singularities and
  multi-index filtrations},  (2007), \url{arXiv:0706.4062}.

\bibitem{CHR}
S.D. Cutkosky, J.~Herzog, and A.~Reguera, \emph{Poincar\'e series of
  resolutions of surface singularities}, Transactions of AMS \textbf{356}
  (2003), no.~5, 1833--1874.

\bibitem{Laufer72}
H.~B. Laufer, \emph{{O}n rational singularities}, Amer. J. Math. \textbf{94}
  (1972), 597--608. \MR{MR0330500 (48 \#8837)}

\bibitem{Line}
A.~N{\'e}methi, \emph{Line bundles associated with normal surface
  singularities}, \url{arXiv:math.AG/0310084}, published in \cite{Graded}.

\bibitem{NSW}
\bysame, \emph{On the {S}eiberg-{W}itten invariants of negative definite
  plumbed 3-manifolds}, manuscript in preparation.

\bibitem{CDGb}
\bysame, \emph{Poincar\'e series associated with surface singularities}, to
  apear in: Singularities I: Algebraic and Analytic Aspects, International
  Conference in Honor of the 60th Birthday of L\^e Dung Tr\'ang, 2007,
  Cuernavaca, Mexico, Contemporary Mathematics, \url{arXiv:0710.0987}.

\bibitem{INV}
\bysame, \emph{Invariants of normal surface singularities}, Real and complex
  singularities, Contemp. Math., vol. 354, Amer. Math. Soc., Providence, RI,
  2004, pp.~161--208.

\bibitem{Graded}
\bysame, \emph{Graded roots and singularities}, Singularities in geometry and
  topology, World Sci. Publ., Hackensack, NJ, 2007, pp.~394--463.
  \MR{MR2311495}

\bibitem{nemethi02:_seiber_witten}
A.~N{\'e}methi and L.I. Nicolaescu, \emph{{S}eiberg-{W}itten invariants and
  surface singularities}, Geom. Topol. \textbf{6} (2002), 269--328
  (electronic). \MR{MR1914570 (2003i:14048)}

\bibitem{NO1}
A.~N{\'e}methi and T.~Okuma, \emph{On the {C}asson invariant conjecture of
  {N}eumann-{W}ahl}, to appear in J. of Algebraic Geometry;
  \url{arXiv:math.AG/0610465}.

\bibitem{NO2}
\bysame, \emph{The {S}eiberg--{W}itten invariant conjecture for
  splice-quotients}, to appear in Journal of London Math. Soc.

\bibitem{Graph}
W.~Neumann, \emph{Graph 3-manifolds, splice diagrams, singularities},
  Singularities in geometry and topology, World Sci. Publ., Hackensack, NJ,
  2007, Proceedings of the Singularity Summer School and Workshop, Trieste,
  August 2005. \MR{MR2311495}

\bibitem{Neu}
W.~D. Neumann, \emph{Abelian covers of quasihomogeneous surface singularities},
  Singularities, Part 2 (Arcata, Calif., 1981), Proc. Sympos. Pure Math.,
  vol.~40, Amer. Math. Soc., Providence, RI, 1983, pp.~233--243. \MR{MR713252
  (85g:32018)}

\bibitem{MR1900786}
W.~D. Neumann and J.~Wahl, \emph{Universal abelian covers of surface
  singularities}, Trends in singularities, Trends Math., Birkh\"auser, Basel,
  2002, pp.~181--190. \MR{MR1900786 (2003c:32028)}

\bibitem{MR1981612}
\bysame, \emph{Universal abelian covers of quotient-cusps}, Math. Ann.
  \textbf{326} (2003), no.~1, 75--93. \MR{MR1981612 (2004d:32039)}

\bibitem{NWuj2}
\bysame, \emph{Complete intersection singularities of splice type as universal
  abelian covers}, Geom. Topol. \textbf{9} (2005), 699--755 (electronic).
  \MR{MR2140991 (2006i:32037)}

\bibitem{NWuj}
\bysame, \emph{Complex surface singularities with integral homology sphere
  links}, Geom. Topol. \textbf{9} (2005), 757--811 (electronic). \MR{MR2140992
  (2006b:32042)}

\bibitem{Opg}
T.~Okuma, \emph{The geometric genus of splice-quotient singularities}, to
  appear in Trans. AMS; \url{arXiv:math.AG/0610464}.

\bibitem{OkumaRat}
\bysame, \emph{Universal abelian covers of rational surface singularities}, J.
  London Math. Soc. (2) \textbf{70} (2004), no.~2, 307--324. \MR{MR2078895
  (2005e:14006)}

\bibitem{Ouac-c}
\bysame, \emph{Universal abelian covers of certain surface singularities},
  Math. Ann. \textbf{334} (2006), no.~4, 753--773. \MR{MR2209255}

\bibitem{Pi1}
H.~Pinkham, \emph{Normal surface singularities with {$C\sp*$} action}, Math.
  Ann. \textbf{227} (1977), no.~2, 183--193. \MR{MR0432636 (55 \#5623)}

\bibitem{CaPl}
C.~Pl{\'e}nat, \emph{\`{A} propos du probl\`eme des arcs de {N}ash}, Ann. Inst.
  Fourier (Grenoble) \textbf{55} (2005), no.~3, 805--823. \MR{MR2149404
  (2006c:14016)}

\bibitem{Ploski}
A.~P{\l}oski, \emph{Newton polygons and the {$\mbox{\L}$}ojasiewicz exponent of
  a holomorphic mapping of {${\bf C}^ 2$}}, Ann. Polon. Math. \textbf{51}
  (1990), 275--281. \MR{MR1093999 (92g:32048)}

\bibitem{Stev}
J.~Stevens, \emph{Universal abelian covers of superisolated singularities},
  \url{arXiv:math.AG/0601669}.

\bibitem{Tei}
B.~Teissier, \emph{Mon\^omes, volumes et multiplicit\'es}, Introduction \`a la
  th\'eorie des singularit\'es, II, Travaux en Cours, vol.~37, Hermann, Paris,
  1988, pp.~127--141. \MR{MR1074593 (92e:14003)}

\bibitem{Wa}
Philip Wagreich, \emph{Elliptic singularities of surfaces}, Amer. J. Math.
  \textbf{92} (1970), 419--454. \MR{MR0291170 (45 \#264)}

\bibitem{Wahl2006}
J.~Wahl, \emph{Topology, geometry, and equations of normal surface
  singularities}, Singularities and computer algebra, London Math. Soc. Lecture
  Note Ser., vol. 324, Cambridge Univ. Press, Cambridge, 2006,
  \url{arXiv:math/050908}, pp.~351--371. \MR{MR2228239 (2007f:32038)}

\end{thebibliography}
\end{document}